\documentclass[12pt,a4paper]{article}
= 0.0in \headheight = 0.0in \topmargin = 0.3in
\def\Dj{\hbox{D\kern-.73em\raise.30ex\hbox{-}
\raise-.30ex\hbox{}}}
\def\dj{\hbox{d\kern-.33em\raise.80ex\hbox{-}
\raise-.80ex\hbox{\kern-.40em}}}
\usepackage{epsfig}
\usepackage{amsmath,amsthm,amsfonts,amssymb,amscd,cite}
\allowdisplaybreaks

\newtheorem{theorem}{Theorem}
\newtheorem{theo}{Theorem}
\newtheorem{lemma}[theo]{Lemma}

\begin{document}

\baselineskip=0.30in

\vspace*{15mm}

\begin{center}
{\Large \bf Resolvent Energy of Unicyclic, Bicyclic \\ and Tricyclic Graphs}

\vspace{10mm}

{\large \bf Luiz Emilio Allem}$^1$, {\large \bf Juliane Capaverde}$^1$,
{\large \bf Vilmar Trevisan}$^1$, \\
{\large \bf Ivan Gutman}$^{2,3}$, {\large \bf Emir Zogi\'c}$^3$,
{\large \bf Edin Glogi\'c}$^3$

\vspace{9mm}

\baselineskip=0.20in

$^1${\it Instituto de Matem\'atica, UFRGS, Porto Alegre, RS, 91509--900, Brazil} \\
{\tt emilio.allem@ufrgs.br, juliane.capaverde@ufrgs.br, \\ trevisan@mat.ufrgs.br} \\[2mm]
$^2${\it Faculty of Science, University of Kragujevac, \\
Kragujevac, Serbia} \\
{\tt gutman@kg.ac.rs} \\[2mm]
$^3${\it State University of Novi Pazar, Novi Pazar, Serbia\/} \\
{\tt  ezogic@np.ac.rs , edinglogic@np.ac.rs}

\vspace{6mm}

(Received December 26, 2015)

\end{center}

\vspace{6mm}

\baselineskip=0.20in

\noindent {\bf Abstract }

\vspace{3mm}

{\small  The resolvent energy of a graph $G$ of order $n$ is
defined as $ER=\sum_{i=1}^n (n-\lambda_i)^{-1}$, where
$\lambda_1,\lambda_2,\ldots,\lambda_n$ are the eigenvalues of
$G$. In a recent work [Gutman et al., {\it MATCH Commun. Math.
Comput. Chem.\/} {\bf 75} (2016) 279--290] the structure of the
graphs extremal w.r.t. $ER$ were conjectured, based on an extensive
computer--aided search. We now confirm the validity of some of
these conjectures.}

\vspace{8mm}

\baselineskip=0.30in

\section{Introduction}

Let $G$ be a graph on $n$ vertices, and let
$\lambda_1 \geq \lambda_2 \geq \cdots \geq \lambda_n$ be its
eigenvalues, that is, the eigenvalues of the adjacency matrix
of $G$. The resolvent energy of $G$ is defined in \cite{gutman,gutman2} as
\begin{equation}\label{resolventenergy}
ER(G)= \sum_{i=1}^n \frac{1}{n-\lambda_i}\,.
\end{equation}

It was shown in \cite{gutman} that
\begin{equation}\label{momre}
ER(G) = \frac{1}{n} \sum_{k=0}^{\infty} \frac{M_k(G)}{n^k}
\end{equation}
where $M_k(G) = \sum_{i=1}^n \lambda_i^k$ is the $k$-th spectral moment of $G$.

In what follows, we present results found in the literature that confirm
some of the conjectures made in \cite{gutman,gutman2} on the resolvent
energy of unicyclic, bicyclic and tricyclic graphs. These results were
originally stated in \cite{MR2900703,MR3280710,MR3128540} in terms of
the Estrada index, another spectrum--based graph invariant related
to the spectral moments by the formula
\begin{equation*}
EE(G) = \sum_{k=0}^{\infty} \frac{M_k(G)}{k!}\,.
\end{equation*}

Most of the proofs in \cite{MR2900703,MR3280710,MR3128540} are based on
the spectral moments and work for the resolvent energy without any change.
The proofs that involve direct calculations with Estrada indices can be
easily modified to give the equivalent results concerning the resolvent
energy, as we show below. Thus, the following are determined:
\begin{itemize}
\item the unicyclic graph with maximum resolvent energy (Theorem \ref{unimax});
\item the unicyclic graphs with minimum resolvent energy (Theorems \ref{unimin} and \ref{newth});
\item the bicyclic graph with maximum resolvent energy (Theorem \ref{bimax});
\item the tricyclic graph with maximum resolvent energy (Theorem \ref{trimax}).
\end{itemize}

\vspace{3mm}

\section{Unicyclic graphs with maximum and minimum resolvent energy}

Let $X_n$ denote the unicyclic graph obtained from the cycle $C_3$ by
attaching $n-3$ pendent vertices to one of its vertices, and $\tilde{X}_n$
the unicyclic graph obtained from $C_4$ by attaching $n-4$ pendent vertices
to one of its vertices, as in Figure \ref{unicyclicmaxi}.

\begin{figure}
\centering
\includegraphics[width=8cm]{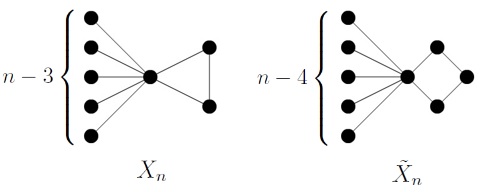}\\
\caption{Unicyclic graphs with maximum resolvent energy.}\label{unicyclicmaxi}
\end{figure}

\newpage

\begin{lemma}\label{maxunim}
Let $G$ be a unicyclic graph on $n \geq 4$ vertices, $G \not\cong X_n, \tilde{X}_n$.
\begin{enumerate}
\item If $G$ is bipartite, then $M_k(G) \leq M_k(\tilde{X}_n)$, for all $k \geq 0$, and $M_{k_0}(G)<M_{k_0}(\tilde{X}_n)$ for some $k_0$.

\item If $G$ is not bipartite (that is, $G$ contains an odd cycle), then $M_k(G) \leq M_k(X_n)$ for all $k\geq 0$, and $M_{k_0}(G)<M_{k_0}(X_n)$ for some $k_0$.
\end{enumerate}
\end{lemma}

\vspace{3mm}

\begin{proof}
Part (i) follows from Lemmas 3.2, 3.5 and 3.8 in \cite{MR2900703}, and (ii)
follows from Lemmas 3.2, 3.5 and 3.7 in \cite{MR2900703}.
\end{proof}

We denote the characteristic polynomial of a graph $G$ by $\phi(G,\lambda)$.
For a proper subset $V_1$ of $V(G)$, $G-V_1$ denotes the graph obtained
from $G$ by deleting the vertices in $V_1$ (and the edges incident on them).
Let $G-v = G-\{v\}$, for $v \in V(G)$. We make use of the following lemma.

\vspace{3mm}

\begin{lemma} {\rm \cite{MR690768}} \label{lemma4.1}
Let $v \in V(G)$, and let $\mathcal{C}(v)$ be the set of cycles
containing $v$. Then
$$
\phi(G,\lambda) = \lambda\,\phi(G-v,\lambda) - \sum_{vw \in E(G)}\phi(G-v-w,\lambda)
-2 \sum_{Z\in \mathcal{C}(v)} \phi(G-V(Z),\lambda)
$$
where $\phi(G-v-w,\lambda) \equiv 1$ if $G$ is a single edge,
and $\phi(G-V{Z},\lambda) \equiv 1$ if $G$ is a cycle.
\end{lemma}

\vspace{3mm}

\begin{theorem}\label{unimax}
Let $G$ be a unicyclic graph on $n \geq 4$ vertices. Then
$ER(G)\leq ER(X_n)$, with equality if and only if $G \cong X_n$.
Moreover, if $G$ is bipartite, then $ER(G) \leq ER(\tilde{X}_n)$,
with equality if and only if $G \cong \tilde{X}_n$.
\end{theorem}

\vspace{3mm}

\begin{proof}
Let $G$ be a unicyclic graph. Lemma \ref{maxunim} and equation \eqref{momre}
imply that $ER(G) \leq ER(X_n)$ if $G$ contains an odd cycle, and
$ER(G) \leq ER(\tilde{X}_n)$ if $G$ contains an even cycle
(i.e., $G$ is bipartite). Furthermore, equality occurs if and only if
$G \cong X_n$, in the case of an odd cycle, or $G \cong \tilde{X}_n$,
in the bipartite case.

Now, an $n$-vertex unicyclic graph with maximum resolvent energy
is either $X_n$ or $\tilde{X}_n$. We show that $ER(X_n)> ER(\tilde{X}_n)$,
for $n\geq 4$. Let $\phi(X_n,\lambda)$ and $\phi(\tilde{X}_n,\lambda)$ denote
the characteristic polynomials of $X_n$ and $\tilde{X}_n$, respectively.
Then, by \cite[Theorem 8]{gutman}, we have
\begin{equation}                    \label{dec25a}
ER(X_n) = \frac{\phi'(X_n,n)}{\phi(X_n,n)}
\hspace{10mm} \mbox{and} \hspace{10mm}
ER(\tilde{X}_n) = \frac{\phi'(\tilde{X}_n,n)}{\phi(\tilde{X}_n,n)}
\end{equation}
where $\phi'(G,\lambda) = \frac{d}{d\lambda}\,\phi(G,\lambda)$.
By Lemma \ref{lemma4.1}, it follows that
\begin{eqnarray*}
\phi(X_n,\lambda)& = & \lambda^{n-4}\,(\lambda^4-n\lambda^2-2\lambda+n-3) \\
\phi(\tilde{X}_n,\lambda) & = & \lambda^{n-4}\,(\lambda^4-n\lambda^2+2n-8)\,.
\end{eqnarray*}
Hence
\begin{eqnarray*}
ER(X_n) - ER(\tilde{X}_n) & = & \frac{\phi'(X_n,n)}{\phi(X_n,n)} - \frac{\phi'(\tilde{X}_n,n)}{\phi(\tilde{X}_n,n)} \\[3mm]
& = & \frac{\phi'(X_n,n)\,\phi(\tilde{X}_n,n)
- \phi'(\tilde{X}_n,n)\,\phi(X_n,n)}{\phi(X_n,n)\,\phi(\tilde{X}_n,n)} \\[3mm]
& = & \frac{10n^4-24n^3+10n^2-4n+16}{(n^4-n^3-n-3)(n^4-n^3+2n-8)}\,.
\end{eqnarray*}
The polynomial $p(\lambda)=10\lambda^4-24\lambda^3+10\lambda^2-4\lambda+16$
does not have any real roots, thus the numerator $p(n)$ is positive
for all $n$. The real roots of the polynomials $\lambda^4-\lambda^3-\lambda-3$
and $\lambda^4-\lambda^3+2\lambda-8$ are less than $2$, so the denominator
is positive for $n \geq 2$. It follows that $ER(X_n) - ER(\tilde{X}_n) > 0$.
\end{proof}

\vspace{3mm}

Let $C_n^*$ denote the unicyclic graph obtained by attaching
a pendent vertex to a vertex of $C_{n-1}$.

\vspace{3mm}

\begin{lemma}\label{minunim}
Let $G$ be a unicyclic graph on $n \geq 5$ vertices.
If $G \not\cong C_n,C_n^*$, then at least one of the
following holds:
\begin{enumerate}
\item $M_k(G) \geq M_k(C_n)$ for all $k\geq 0$, and $M_k(G)> M_k(C_n)$ for some $k_0 \geq 0$.
\item $M_k(G) \geq M_k(C_n^*)$ for all $k\geq 0$, and $M_k(G)> M_k(C_n^*)$ for some $k_0 \geq 0$.
\end{enumerate}
\end{lemma}

\begin{proof}
If $G \not\cong C_n,C_n^*$, it follows from Lemmas 5.3, 5.4, 5.5 in \cite{MR2900703} that $G$ can be transformed into either $C_n$ or $C_n^*$ in a finite number of steps, in such a way that, at each step, the $k$-th spectral moment does not increase for each $k$, and decreases for some $k_0$.
\end{proof}

\begin{theorem}\label{unimin}
Let $G$ be a unicyclic graph on $n \geq 5$ vertices. If $G \not\cong C_n,C_n^*$, then $ER(G)> \min\{ER(C_n),ER(C_n^*)\}$.
\end{theorem}

\begin{proof}
Follows from Lemma \ref{minunim}.
\end{proof}

Using arguments that are not based on spectral moments, we can strengthen Theorem \ref{unimin}
as follows:

\begin{theorem} \label{newth}
Let $G$ be a unicyclic graph on $n \geq 5$ vertices. If $G \not\cong C_n$, then $ER(G) > ER(C_n)$.
\end{theorem}

\begin{proof}
In view of Theorem \ref{unimin}, it is sufficient to prove that $ER(C_n) < ER(C_n^*)$.
In \cite{gutman}, the validity of this latter inequality was checked for $n \leq 15$.
Therefore, in what follows we may asume that $n > 15$, i.e., that $n$ is sufficiently
large.

Bearing in mind the relations (\ref{dec25a}), we get
$$
ER(G) = \left. \frac{d\ln \phi(G,\lambda)}{d\lambda} \right|_{\lambda=n}
$$
and therefore
\begin{eqnarray}
ER(C_n) - ER(C_n^*) & = & \left.  \left( \frac{d\ln \phi(C_n,\lambda)}{d\lambda}  -
\frac{d\phi(\ln C_n^*,\lambda)}{d\lambda}  \right) \right|_{\lambda=n} \nonumber \\[3mm]
& = & \left. \frac{d}{d\lambda} \ln \frac{\phi(C_n,\lambda)}{\phi(C_n^*,\lambda)}
\right|_{\lambda=n}\,. \label{dec25b}
\end{eqnarray}
The greatest eigenvalue of $C_n$ is 2, and the greatest eigenvalue of $C_n^*$ is
certainly less than 3. Therefore, bearing in mind that $C_n$ and $C_n^*$ contain no
triangles and no pentagons, for $\lambda=n$,
\begin{eqnarray*}
\phi(C_n,\lambda) & = & \lambda^n - n\,\lambda^{n-2} + b_2(C_n)\,\lambda^{n-4} + \cdots \\
\phi(C_n^*,\lambda) & = & \lambda^n - n\,\lambda^{n-2} + b_2^*(C_n)\,\lambda^{n-4} + \cdots
\end{eqnarray*}
and thus
$$
\frac{\phi(C_n,\lambda)}{\phi(C_n^*,\lambda)} = 1 + \frac{b_2(C_n) - b_2(C_n^*)}{\lambda^4}
+ O\left( \frac{1}{\lambda^6} \right)
$$
and
$$
\ln \frac{\phi(C_n,\lambda)}{\phi(C_n^*,\lambda)} = \frac{b_2(C_n) - b_2(C_n^*)}{\lambda^4}
+ O\left( \frac{1}{\lambda^6} \right) .
$$
Then because of (\ref{dec25b}),
\begin{equation}                    \label{dec25c}
ER(C_n) - ER(C_n^*) = -4\,\frac{b_2(C_n) - b_2(C_n^*)}{n^5} + O\left( \frac{1}{n^7} \right)\,.
\end{equation}

Using the Sachs coefficient theorem \cite{MR690768}, one can easily show that
$$
b_2(C_n) = \frac{1}{2}\,n(n-3) \hspace{10mm} \mbox{and} \hspace{10mm}
b_2(C_n^*) = \frac{1}{2}\,(n-3)(n-4) + 2n - 7
$$
from which one immediately gets that for sufficiently large values of $n$,
$$
ER(C_n) - ER(C_n^*) \approx -\frac{4}{n^5}
$$
i.e., $ER(C_n) < ER(C_n^*)$.
\end{proof}

\vspace{3mm}

\section{Bicyclic graphs with maximum resolvent energy}

Let $\theta(p,q,\ell)$ be the union of three internally disjoint
paths $P_{p+1}, P_{q+1}, P_{\ell+1}$ with common end vertices.
Let $Y_n$ denote the bicyclic graph obtained from $\theta(2,2,1)$ by
attaching $n-4$ pendent vertices to one of its vertices of degree 3,
and let $\tilde{Y}_n$ denote the bicyclic graph obtained from
$\theta(2,2,2)$ by attaching $n-5$ pendent vertices to one of its
vertices od degree 3, as in Figure \ref{bicyclics}.

\vspace{5mm}

\begin{figure}[h]
  \centering
  \includegraphics[width=10cm]{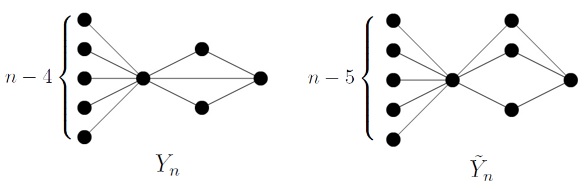}\\
  \caption{Bicyclic graphs with maximum resolvent energy.}\label{bicyclics}
\end{figure}

\newpage

\begin{lemma}\label{maxbim} Let $G$ be a bicyclic graph on $n \geq 5$ vertices, $G \not\cong Y_n, \tilde{Y}_n$. Then one of the following holds:
\begin{enumerate}
\item $M_k(G) \leq M_k(Y_n)$ for all $k\geq 0$, and $M_k(G)< M_k(Y_n)$ for some $k_0 \geq 0$.
\item $M_k(G) \leq M_k(\tilde{Y}_n)$ for all $k\geq 0$, and $M_k(G)< M_k(\tilde{Y}_n)$ for some $k_0 \geq 0$.
\end{enumerate}
\end{lemma}

\vspace{3mm}

\begin{proof}
Follows from Lemma 3.1, Theorems 3.2 and 3.3, and Lemma 3.4 in \cite{MR3280710}.
\end{proof}

\begin{theorem}\label{bimax}
Let $G$ be a bicyclic graph on $n \geq 5$ vertices. Then
$ER(G)\leq ER(Y_n)$, with equality if and only if $G \cong Y_n$.
\end{theorem}

\vspace{3mm}

\begin{proof}
By Lemma \ref{maxbim}, a graph with maximum resolvent energy among
$n$-vertex bicyclic graphs is either $Y_n$ or $\tilde{Y}_n$.
Thus, it is sufficient to show that $ER(Y_n)> ER(\tilde{Y}_n)$,
for $n\geq 5$. Let $\phi(Y_n,\lambda)$ and $\phi(\tilde{Y}_n,\lambda)$
denote the characteristic polynomials of $Y_n$ and $\tilde{Y}_n$,
respectively. Then, by \cite[Theorem 8]{gutman}, we have
$$
ER(Y_n) = \frac{\phi'(Y_n,n)}{\phi(Y_n,n)} \hspace{10mm} \mbox{and} \hspace{10mm}
ER(\tilde{Y}_n) = \frac{\phi'(\tilde{Y}_n,n)}{\phi(\tilde{Y}_n,n)}\,.
$$
By Lemma \ref{lemma4.1}, it follows that
\begin{eqnarray*}
\phi(Y_n,\lambda)& = & \lambda^{n-4}\,\big[\lambda^4-(n+1)\lambda^2 - 4\lambda + 2(n-4) \big] \\
\phi(\tilde{Y}_n,\lambda)& = & \lambda^{n-4}\,\big[\lambda^4-(n+1)\lambda^2+3(n-5) \big]\,.
\end{eqnarray*}
Hence
\begin{eqnarray*}
ER(Y_n) - ER(\tilde{Y}_n) & = & \frac{\phi'(Y_n,n)}{\phi(Y_n,n)} - \frac{\phi'(\tilde{Y}_n,n)}{\phi(\tilde{Y}_n,n)} \\[3mm]
& = & \frac{\phi'(Y_n,n)\,\phi(\tilde{Y}_n,n)
-\phi'(\tilde{Y}_n,n)\,\phi(Y_n,n)}{\phi(Y_n,n)\,\phi(\tilde{Y}_n,n)} \\[3mm]
& = & \frac{16n^4-34n^3+8n^2+2n+60}{(n^4-n^3-n^2-2n-8)(n^4-n^3-n^2+3n-15)}\,.
\end{eqnarray*}
The polynomial $p(x)=16x^4-34x^3+8x^2+2x+60$ does not have any real roots,
thus the numerator $p(n)$ is positive for all $n$. The real roots of the
polynomials $x^4-x^3-x^2-2x-8$ and $x^4-x^3-x^2+3x-15$ are less than 3,
so the denominator is positive for $n \geq 3$. It follows that
$ER(Y_n) - ER(\tilde{Y}_n) > 0. $
\end{proof}

\vspace{3mm}

\section{Tricyclic graphs with maximum resolvent energy}

Let $Z_n^i$, $1 \leq i \leq 6$, be the graphs given in Figure \ref{tricyclics}.
\begin{figure}[h]
\centering
\includegraphics[width=15cm]{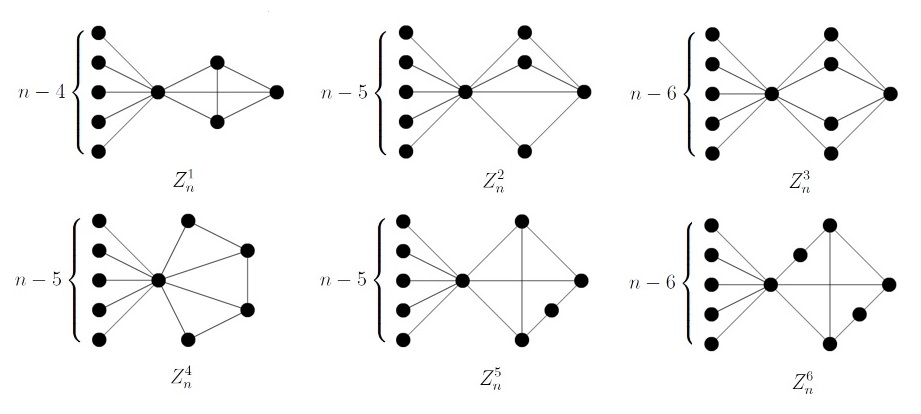}\\
\caption{Tricyclic graphs with maximum resolvent energy.}\label{tricyclics}
\end{figure}

\begin{lemma}\label{maxtrim}
Let $G$ be a bicyclic graph on $n \geq 4$ vertices such that
$G \not\cong Z_n^i$, for $1 \leq i \leq 6$. Then, some
$i \in \{1,2,3,4,5,6\}$, $M_k(G) \leq M_k(Z_n^i)$ for all
$k \geq 0$, and $M_k(G)< M_k(Z_n^i)$ for some $k_0 \geq 0$.
\end{lemma}
\begin{proof}
Follows from Corollaries 3.5 and 3.9 and Lemma 3.10 in \cite{MR3128540}.
\end{proof}

\vspace{3mm}

\begin{theorem}\label{trimax}
Let $G$ be a tricyclic graph on $n\geq 4$ vertices. Then
$ER(G) \leq ER(Z_n^1)$, with equality if and only if $G \cong Z_n^1$.
\end{theorem}

\vspace{3mm}

\begin{proof} By Lemma \ref{maxtrim}, a graph with maximum resolvent
energy among $n$-vertex tricyclic graphs is equal to $Z_n^i$, for
some $1 \leq i \leq 6$.

By Lemma \ref{lemma4.1}, it follows that
\begin{eqnarray*}
\phi(Z_n^1,\lambda)& = & \lambda^{n-5}\,
(\lambda^5-(n+2)\lambda^3-8\lambda^2+3(n-5)\lambda+2(n-4)) = \lambda^{n-5}\,f_1(\lambda) \\
\phi(Z_n^2,\lambda)& = & \lambda^{n-4}\,
(\lambda^4-(n+2)\lambda^2-6\lambda+3(n-5))=\lambda^{n-4}\,f_2(\lambda) \\
\phi(Z_n^3,\lambda)& = & \lambda^{n-4}\,
(\lambda^4-(n+2)\lambda^2+4(n-6))=\lambda^{n-4}\,f_3(\lambda) \\
\phi(Z_n^4,\lambda)& = & \lambda^{n-6}\,
(\lambda^6-(n+2)\lambda^4-6\lambda^3+3(n-4)\lambda^2+2\lambda-(n-5))=\lambda^{n-6}\,f_4(\lambda) \\
\phi(Z_n^5,\lambda)& = & \lambda^{n-5}\,
(\lambda^5-(n+2)\lambda^3-4\lambda^2+4(n-4)\lambda+4)=\lambda^{n-5}\,f_5(\lambda) \\
\phi(Z_n^6,\lambda)& = & \lambda^{n-6}\,
(\lambda^6-(n+2)\lambda^4+5(n-5)\lambda^2-2(n-8))=\lambda^{n-6}\,f_6(\lambda)\,.
\end{eqnarray*}

For $2\leq i\leq 6$, we have
\begin{eqnarray*}
ER(Z_n^1) - ER(Z_n^i) & = & \frac{\phi'(Z_n^1,n)}{\phi(Z_n^1,n)}
- \frac{\phi'(Z_n^i,n)}{\phi(Z_n^i,n)} \\[3mm]
& = & \frac{\phi'(Z_n^1,n)\,\phi(Z_n^i,n)
- \phi'(Z_n^i,n)\,\phi(Z_n^1,n)}{\phi(Z_n^1,n)\,\phi(Z_n^i,n)}\,.
\end{eqnarray*}
Straightforward calculation yields
\begin{eqnarray*}
ER(Z_n^1) - ER(Z_n^2) & = & \frac{6n^6-12n^5+42n^4-18n^3-42n-120}{n\,f_1(n)\,f_2(n)} \\[3mm]
ER(Z_n^1) - ER(Z_n^3) & = & \frac{28n^6-56n^5+44n^4-8n^3+136n^2+80n-192}{n\,f_1(n)\,f_3(n)} \\[3mm]
ER(Z_n^1) - ER(Z_n^4) & = & \frac{6n^8+40n^6-2n^5-48n^4-96n^3-188n^2-132n-40}{n\,f_1(n)\,f_4(n)} \\[3mm]
ER(Z_n^1) - ER(Z_n^5) & = & \frac{16n^7-20n^6+56n^5-40n^4+4n^3-52n^2-188n}{n\,f_1(n)\,f_5(n)} \\[3mm]
ER(Z_n^1) - ER(Z_n^6) & = &
\frac{32n^8\!-\!62n^7\!+\!30n^6\!+\!92n^5\!+\!94n^4\!-\!2n^3\!-\!
432n^2\!-\!432n\!-\!128}{n\,f_1(n)\,f_6(n)}\,.
\end{eqnarray*}

All the real roots of the polynomials that appear in the numerators are
less than 2. Moreover, all the real roots of the polynomials $f_i$, $1\leq i \leq 6$,
are less than 3. It follows that the numerator end denominator in the quotients
above are positive for $n\geq 3$. Hence
$ER(Z_n^1) - ER(Z_n^i)>0$ for $2 \leq i \leq 6$.
\end{proof}

\end{document}